\documentclass[leqno,12pt]{amsart} 

\usepackage[utf8x]{inputenc} 
\usepackage{amsmath, nicefrac, a4wide, mathtools}
\usepackage{amssymb}
\usepackage{amsfonts}
\usepackage{amscd}
\usepackage{enumerate}

\usepackage{hyperref}

\theoremstyle{plain} 
\newtheorem{theorem}{\indent\sc Theorem}[section]
\newtheorem{lemma}[theorem]{\indent\sc Lemma}
\newtheorem{corollary}[theorem]{\indent\sc Corollary}
\newtheorem{proposition}[theorem]{\indent\sc Proposition}

\theoremstyle{definition} 

\newtheorem{remark}[theorem]{\indent\sc Remark}

\newtheorem{example}[theorem]{\indent\sc Example}

\newtheorem{note}[theorem]{\indent\sc Note}

\newcommand{\LieSO}{{\mathrm{SO}}}

\newcommand{\LieSp}{{\mathrm{Sp}}}

\newcommand{\LieU}{{\mathrm{U}}}

\newcommand{\lieA}{{\mathfrak{a}}}

\newcommand{\lieG}{{\mathfrak g}}

\newcommand{\lieK}{{\mathfrak k}}

\newcommand{\lieS}{{\mathfrak s}}

\newcommand{\lieP}{{\mathfrak p}}

\newcommand{\lieT}{{\mathfrak t}}

\newcommand{\Ad}{{\mathrm{Ad}}}

\newcommand{\ad}{{\mathrm{ad}}}

\newcommand{\Exp}{{\mathrm{Exp}}}

\newcommand{\Hom}{{\mathrm{Hom}}}

\newcommand{\Root}{{\mathcal{R}}}

\DeclareMathOperator{\aff}{aff}

\def\R{{\mathbb R}}
\def\Sphere{{S}}

\def\Z{{\mathbb Z}}
\def\C{{\mathbb C}}

\title{Symmetric Spaces with Rectangular Unit Lattices, Revisited}
\author[Eschenburg]{Jost-Hinrich Eschenburg}
\email{jost-hinrich.eschenburg@math.uni-augsburg.de}
\author[Heintze]{Ernst Heintze}
\email{ernst.heintze@math.uni-augsburg.de}
\author[Quast]{Peter Quast}
\email{peter.quast@math.uni-augsburg.de}
\address{Institut f\"ur Mathematik, Universit\"at Augsburg, 86135 Augsburg, Germany}

\date{\today}



\begin{document}

\begin{abstract}
We give a new proof of a theorem of \textsc{Loos} stating that a Riemannian symmetric space $X$ with rectangular unit lattice  is a symmetric $R$-space. For this we  construct explicitly an isometric extrinsically symmetric embedding of $X$ in a Euclidean space which reveals $X$ as a standardly embedded symmetric $R$-space. We further determine the root systems, Euclidean root data, fundamental groups
 and eigenvalues of the Laplacian of symmetric spaces with rectangular unit lattice in a direct way.
\end{abstract}

\maketitle
\section{Introduction}

This work grew out from an attempt to understand a fundamental theorem of \textsc{Loos} \cite[Satz 6]{LoosCubic}
that establishes a one-to-one correspondence between symmetric $R$-spaces and compact Riemannian symmetric spaces with cubic unit lattice. Recall that a \emph{symmetric $R$-space} arises as an orbit $L(\xi)$ of a point $\xi$ in the boundary at infinity of a Riemannian symmetric space $L/G$ of non-compact type ($L$ semi-simple, non-compact and with trivial center and $G$ maximal compact in $L$) such that $(G,G_\xi)$
is a symmetric pair, where $G_{\xi}$ is the stabilizer of $\xi$ in $G.$ Since $G(\xi)=L(\xi),$ the symmetric $R$-space can be identified with $G/G_\xi$ and is thus a compact affine symmetric space. The \emph{unit lattice} of a compact Riemannian symmetric space $X$ is the lattice of a maximal torus $T$ of $X$
that is the set of all tangent vectors of $T$ at a point $p\in T$ which are mapped to $p$ by the Riemannian exponential map at $p.$
It is called
\emph{rectangular (cubic)} if it has an orthogonal basis (orthonormal basis after scaling). Then the main result of \textsc{Loos} \cite[Satz 5,6]{LoosCubic} (announced in \cite{LoosAnnounce}) says more precisely that a Riemannian symmetric space with cubic unit lattice is affinely equivalent to a symmetric $R$-space and that moreover a symmetric $R$-space $G/G_{\xi}$ has (up to a scaling factor) a unique $G$-equivariant metric turning it into a Riemannian symmetric space with cubic unit lattice.
\textsc{Loos}' proof consists of translating the problem into the language of Jordan triple systems where he solves it using extensive  calculations.
We therefore were looking for a second, possibly more geometric proof. It is led by the following ideas: The unit sphere in the tangent space
of $L/G$ at $eG$ can be identified with the boundary at infinity of $L/G$ in a $G$-equivariant way by mapping $v$ to $\gamma_{v}(\infty),$ where $\gamma_{v}$ is the geodesic in $L/G$ starting at $eG$ in direction $v.$ A symmetric $R$-space $G/G_{\xi}$ can thus be identified with an orbit $G(v)$ of the linear isotropy representation ($s$-representation) of $L/G$ and inherits a Riemannian metric.
This is called the \emph{standard embedding} of the symmetric $R$-space $G/G_{\xi}.$
\textsc{Ferus} \cite{Feru-80} has shown that a standardly embedded symmetric $R$-space is an \emph{extrinsically symmetric space} (that is
a submanifold of a Euclidean space which is invariant under the reflections along all its affine normal spaces) and, more importantly, that any compact extrinsically symmetric space of positive dimension is congruent to a symmetric $R$-space considered as an $s$-orbit (see \cite{EH-95} for a short proof).
Now our main result can be formulated as follows:

\begin{theorem}\label{THM: A}
A Riemannian symmetric space $X$ with rectangular unit lattice admits a full isometric embedding $\Phi$ in a Euclidean space $E$ whose
image is extrinsically symmetric. Moreover, such an  embedding is unique up to congruence.
\end{theorem}

Note that the converse of the first part of Theorem \ref{THM: A} is also true (see \cite{EHQ} and Theorem \ref{THM: Clifford type Extr Sym}): The unit lattice of a compact extrinsically symmetric submanifold is rectangular.
By means of \textsc{Ferus}' result the first part of Theorem \ref{THM: A}
shows that a Riemannian symmetric space with rectangular unit lattice is affinely equivalent to a symmetric $R$-space and thus provides a new proof of the main result of \textsc{Loos}. As a direct consequence of the uniqueness statement of  Theorem \ref{THM: A}  every (intrinsic) isometry between full compact extrinsically symmetric spaces
is induced from an isometry between the ambient Euclidean spaces (Corollary \ref{COR: Extension of isometries}).\par

Our proof of Theorem \ref{THM: A} is based on the following considerations: Let $X=G/K$ be a compact Riemannian symmetric space
with rectangular unit lattice. Since compact extrinsically symmetric spaces are essentially $s$-orbits that are symmetric $R$-spaces,
we are looking for a $G$-equivariant embedding $\Phi:X\to E$ in a Euclidean vector space $E$ on which $G$ acts by orthogonal transformations. Such an embedding $\Phi$ maps $eK$ to a non-zero element $v_o\in E$ which is fixed by $K.$
An irreducible representation of $G$ with a non-zero $K$-fixed vector is called \emph{$K$-spherical} and may be seen as an irreducible representation of $X.$ According to a theorem of \textsc{Cartan}, \textsc{Helgason} and \textsc{Takeuchi} (see \cite{TakeBook} and Appendix \ref{SEC: Cartan-Helgason-Thm}) the complex irreducible representations of $X=G/K$ are parameterized by orbits of the Weyl group $W_X$ of $X$ acting on the dual $\Gamma_X^*$ of the the unit lattice $\Gamma_X$ of $X,$ in complete analogy to the highest weight theorem for representations of compact Lie groups. If $X$ is indecomposable (that is not a Riemannian product of two symmetric spaces of positive dimension) and $\varepsilon_1,\dots, \varepsilon_r\in\Gamma_X^*$ is the dual basis of an orthogonal basis of $\Gamma_X,$ then we show
that $W_X(\varepsilon_1)$ contains $\{\varepsilon_1,\dots,\varepsilon_r\}$ (Proposition \ref{PROP: Weyl orbit}). It is therefore tempting
to consider the complex irreducible $K$-spherical $G$-module $V$ with the lowest possible `highest weight' $W_X(\varepsilon_1).$ If
$v_0\in V$ is a non-zero fixed vector of $K$ and $V$ is endowed with the real part of a suitable $G$-invariant hermitian metric, then
$\Phi:X=G/K\to V,\; gK\mapsto gv_0,$ turns out to be an isometric embedding. Moreover, since there are no non-zero weights strictly lower than $\varepsilon_1,$ $\Phi$ maps maximal tori of $X$ to Clifford tori in $V,$ that is to products of circles in pairwise orthogonal planes. By a previous result of the authors \cite{EHQ} (see Theorem
\ref{THM: Clifford type Extr Sym}) $\Phi(X)$ is extrinsically symmetric.\par

To prepare the proof as well as to investigate Riemannian symmetric spaces with rectangular unit lattice by themselves we study in the first part of this paper the relationship between the unit lattice and the (restricted) root system of a compact Riemannian symmetric space. Several of the results are known. However, our proofs are more direct neither using the correspondence between symmetric spaces with rectangular unit lattice and symmetric $R$-spaces nor classification results. We formalize the relationship between unit lattices and root systems by defining the notion of a \emph{Euclidean root datum} of a compact Riemannian symmetric space generalizing the root data of compact Lie groups (equivalently of complex reductive algebraic groups). A Euclidean root datum not only contains the information about the root system, but also the position of the symmetric space within the poset  of symmetric quotients of its universal cover and, in particular, its fundamental group. We compute explicitly the Euclidean root data for indecomposable symmetric spaces with rectangular unit lattice and show that there are precisely five possibilities (Theorem \ref{THM: possible root datum types}). As applications we get a generalization of the polysphere theorem for hermitian symmetric spaces of compact type to all simply connected symmetric spaces with rectangular unit lattice (Theorem \ref{THM: Polysphere}) and an explicit formula for the eigenvalues of the Laplacian for this class of spaces (Proposition \ref{PROP: lambda-omega}). The last result implies in particular that the extrinsically symmetric embeddings $\Phi$ from Theorem \ref{THM: A} are embeddings in the first eigenspace of the Laplacian in almost all cases (Corollary \ref{COR: embedding into eigenspaces of Laplacian}), confirming a result of \textsc{Ohnita} \cite{Ohnita} about standard embeddings of symmetric $R$-spaces.


\section{Preliminaries}
\label{SEC: preliminaries}

We recall some standard facts about symmetric spaces and fix notations. For more details see e.\ g.\ \cite{Helg-78, Loos2}.\par
In the following a symmetric space will always mean a connected Riemannian symmetric space of positive dimension, usually denoted by $X.$ Thus for each $x\in X$ there exists an involutive isometry $s_x$ of $X,$ called the \emph{geodesic symmetry} at $x$,
that has $x$ as an isolated fixed point.
We denote by $I(X)$ the group of isometries of $X$ and by
$G=G(X)$ the closed subgroup of $I(X)$ that is generated by the transvections $s_p\circ s_q$ with $p,q\in X.$
Then $G$ is a connected Lie group that acts transitively and effectively on $X$
and we call it the  \emph{transvection group} of $X.$
If $X$ is compact, then $G$ coincides with
the connected component of $I(X)$ containing the identity.\par
We now fix a base point $p\in X$ and let $K=G_p$ be the isotropy group of $G$ at $p.$ This yields an identification
$G/K\to X,\; gK\mapsto g.p:=g(p),$ which becomes an isometry if $G/K$ is endowed with the pull-back metric. This is $G$-invariant and the
geodesic symmetry $s_p$ of $X$  at $p$ corresponds on $G/K$ to the map $gK\mapsto \sigma(g)K,$ where $\sigma$ is the involutive
automorphism of $G$ given by the conjugation with $s_{p}.$ Since $K$ is an open subgroup of $G^{\sigma}=\{g\in G:\; \sigma(g)=g\},$
$(G,K)$ is a symmetric pair.
The differential of $\sigma$ at $e\in G,$ also denoted by $\sigma,$
induces a splitting of the Lie algebra $\lieG$ of $G$ as $\lieG=\lieK\oplus\lieP$ into the fixed point sets $\lieK$
and $\lieP$ of $\sigma$ and $-\sigma,$ respectively. Note that $\lieK$ is the Lie algebra of $K.$ We identify $\lieP$ with $T_{p}X$ via
$v\mapsto \left.\frac{d}{dt}\right|_{t=0}(\exp(tv).p),$
where $\exp:\lieG\to G$ is the Lie theoretic exponential map,
and we take  on $\lieP$ the pull-back inner product.
Using this identification the curvature tensor $R$ of $X$ at $p$ is given by
$R(X,Y)Z=[Z,[X,Y]]$
for all $X,Y,Z\in \lieP$ and the Riemannian exponential map
$\Exp_p:T_pX\to X$ of $X$ at $p$ satisfies $\Exp_p(V)=\exp(V).p$ for all $V\in\lieP.$\par

We further fix a  {maximal flat} $F_X$ of $X$ through $p,$ that is a connected, flat, totally geodesic submanifold of $X$ of maximal dimension.
Since $G$ acts transitively on the set of pairs $(x,F),$ where $x\in X$ and $F$ is a maximal flat of $X$ containing $x,$
 the \emph{rank} of $X,$ which is the dimension of a maximal flat of $X,$ is well-defined.
The tangent space of $F_X$ at $p$
is identified with a maximal abelian subspace $\lieA$ of $\lieP.$ Conversely, any maximal abelian subspace $\lieA$ of $\lieP$ gives rise to
the maximal flat $\exp(\lieA).p$ of $X$ through $p.$ \par
The \emph{unit lattice} of a symmetric space $X$ (with respect to $(p,F_X)$) is
$$\Gamma_X:=\{H\in\lieA:\; \exp(H).p=p\}=\{H\in\lieA:\; \exp(H)\in K\},$$
which is  a \emph{lattice} in $\lieA,$ that is a discrete subgroup of the additive group $\lieA.$
$\Gamma_X$ is full in $\lieA,$ that is it spans $\lieA,$ if and only if
 $X$ is compact. In any case the Riemannian exponential map identifies
$\lieA/\Gamma_X$ with $F_X.$
If $X$ is compact $\lieA/\Gamma_X$ is a flat torus and $F_X$ is called a \emph{maximal torus} of $X,$ also  denoted by $T_X.$\par

Assume from now on that $X$ is compact or, more generally, covers a compact symmetric space. For an element $\alpha\in\lieA^*=\Hom(\lieA,\R)$ let
$$\lieG_{\alpha}:=\{Z\in\lieG\otimes\C: \ad(H)Z=2\pi i\alpha(H)Z\; \text{for all}\;  H\in\lieA\}$$ and
$$
 \Root_X:=\big\{\alpha\in \lieA^*\setminus\{0\}:\; \lieG_{\alpha}\neq\{0\}\big\}.
$$
Then
$[\lieG_\alpha,\lieG_\beta]\subset\lieG_{\alpha+\beta}$ for all $\alpha,\beta\in\lieA^*$ and
$$\lieG\otimes\C=\lieG_0\oplus\bigoplus_{\alpha\in\Root_X}\lieG_{\alpha}.$$
$\Root_X$ is called the \emph{root system} of $X$ (with respect to $(p,F_X)$) although it is only a root system in a subspace of $\lieA^*.$ Note that up to isomorphism the triple $(\lieA,\Gamma_X,\Root_X)$
 only depends on $X,$ that is if $(p,F)$ is replaced by $(p',F')$ and thus $\lieA,\; \Gamma_X,\; \Root_X$ by
 $\lieA',\; \Gamma'_X,\; \Root'_X$ then there exists a linear isometry $\varphi:\lieA\to\lieA'$ with $\varphi(\Gamma_X)=\Gamma'_X$ and
 $\varphi(\Root_X)=\Root'_X$ where  $\varphi(\alpha)=\alpha\circ \varphi^{-1}$ for $\alpha\in\lieA^*.$
The interplay between the unit lattice and the set of roots is discussed in Appendix \ref{APPENDIX: unit lattice of compact sym space}.\par
If $\pi:X\to X'$ is a Riemannian covering of symmetric spaces then $G(X)$ covers $G(X')$ as the geodesic symmetries of $X$
push down to geodesic symmetries of $X'$ inducing a surjective homomorphism $G=G(X)\to G(X')=G'.$ Its kernel is contained in the group of deck transformations of $\pi$ and thus discrete. We may assume $\lieG=\lieG'$ for the Lie algebras of $G$ and $G'$ and further $\lieK=\lieK'$
and $\lieA=\lieA'$ by choosing appropriate base points and maximal flats in $X$ and $X'.$ In particular
$\Gamma_X\subset\Gamma_{X'}$ and $\Root_X=\Root_{X'}.$
Since the universal cover of $X$ splits as $X_0\times X_1,$ where $X_0$ is a Euclidean space and $X_1$ is a simply connected symmetric space of compact type (one of them possibly a point), $\lieA$ splits orthognally as $\lieA=\lieA_0\oplus \lieA_1$ with $\lieA_0=\bigcap\limits_{\alpha\in\Root_X} \mathrm{ker}(\alpha).$ Restriction to $\lieA_1$ yields an isometric isomorphism from
$\mathrm{span}(\Root_X)$ to $\lieA_1^*$ that maps $\Root_X$ onto the root system of $X_1.$ In particular, $\Root_X$ is a root system in its linear span. \par
The  \emph{Weyl group} $W_X$ of $X$ (corresponding to $(p,F)$)
is the subgroup of the orthogonal group of $\lieA$ that is generated by the orthogonal reflections $s_\alpha$ along the hyperplanes
$\mathrm{ker}(\alpha)$ with $\alpha\in\Root_X.$
Recall that $$s_{\alpha}(H)=H-\alpha(H)\alpha^\vee$$ for all $H\in\lieA,$ where $\alpha^\vee\in\lieA$ is the unique vector orthogonal to
$\mathrm{ker}(\alpha)$ with $\alpha(\alpha^\vee)=2.$ It is equal to $\frac{2}{\|H_\alpha\|^2}H_\alpha,$ where
$H_\alpha\in\lieA$ is determined by $\alpha(H)=\langle H_{\alpha}, H\rangle$ for all $H\in\lieA.$ $\Root^\vee_X:=\{\alpha^\vee:\; \alpha\in\Root_X\}$
is called the \emph{set of inverse roots} and is a root system in $\lieA_0^\perp.$
The Weyl group has a second description as $W_X=M'/M$
with
$M'=\{k\in K:\;  \Ad(k)\lieA=\lieA\}$ and
$M=\{k\in K:\;  \Ad(k)H=H\; \text{for all}\; H\in\lieA\}$
 (see e.g.\ \cite[Prop.\ 2.2, p.\ 67]{Loos2}).
In particular, $W_X$ preserves the unit lattice $\Gamma_X$
and the set of roots $\Root_X.$
 Note that $W_X$ also acts on $\lieA^*$ by $w.\varphi:=\varphi\circ w^{-1}$
 for all $w\in W_X$ and $\varphi\in\lieA^*.$\par
A symmetric space $X$ is called \emph{decomposable} if $X$ is isometric to a
Riemannian product of positive dimensional Riemannian manifolds. Otherwise $X$ is called \emph{indecomposable}.

\begin{lemma}[Splitting Lemma for compact symmetric spaces]
\label{LEM: splitting}
Let $X\cong G/K$ be a compact symmetric space. Then the following assertions are equivalent:
\begin{enumerate}[(1)]
 \item $X$ is decomposable.
 \item $\lieA$ splits
 as an orthogonal direct sum $\lieA=\lieA_1\oplus\lieA_2$  of non-trivial $W_X$-invariant subspaces such that
 the unit lattice $\Gamma_X$ of $X$ splits accordingly as $\Gamma_X=(\Gamma_X\cap\lieA_1)+
 (\Gamma_X\cap\lieA_2).$
\end{enumerate}
\end{lemma}
\begin{proof}
Obviously the first statement implies the second one. We now assume the second assertion.
Since for all $\alpha\in\Root_X$ the reflection
$s_\alpha$ is an element of $W_X,$
$\lieA_1$ and $\lieA_2$ split into eigenspaces of $s_\alpha.$
The $(-1)$-eigenspace of such an $s_\alpha$ is generated by $H_\alpha\in\lieA.$
Thus for any $\alpha\in\Root_X$ we have either $H_\alpha\in\lieA_1$
or $H_\alpha\in\lieA_2.$ Hence the set of roots of $X$ is a disjoint union
$\Root_X=\Root_1\cup\Root_2$ where
$\Root_j=\{\alpha\in\Root_X:\; H_\alpha\in \lieA_j\},\; j=1,2$ (cf.\ \cite[Prop.\ 10, p.\ V-21]{Serre}).
Thus
$\lieP\cong T_oX$ splits orthogonally into two curvature invariant subspaces
$\lieP=\lieP_1\oplus\lieP_2,$ where
$\lieP_j=\lieA_j+\sum_{\alpha\in\Root_j}(\lieG_\alpha+\lieG_{-\alpha})\cap\lieP,\; j=1,2.$
Consequently $X$ splits locally as $X_1\times X_2,$ where
$X_j=\exp(\lieP_j).p,\; j\in\{1,2\},$ are totally geodesic submanifolds of $X$
(see e.g.\ \cite[Chap.\ IV, \S 7]{Helg-78}), and $X$ is covered by $X_1\times X_2.$\par
On the other hand, the maximal torus of $X$ splits globally as
$T_X =\exp(\lieA).p\cong\lieA/\Gamma_X=(\lieA_1+\lieA_2)/((\Gamma_X\cap\lieA_1)+(\Gamma_X\cap\lieA_2))\cong
(\lieA_1/(\Gamma_X\cap\lieA_1))\times (\lieA_2/(\Gamma_X\cap\lieA_2))=T_1\times T_2$
where $T_j = \exp(\lieA_j).p\cong\lieA_j/(\Gamma_X\cap \lieA_j)$ for $j=1,2.$
Thus when restricted to $T_X =T_1\times T_2$
(or to any other maximal torus of $X_1\times X_2$) the covering $\pi:X_1\times X_2 \to X$ is one-to-one.
Let now $q$ and $q'$ be two points in $X_1\times X_2$ with
$\pi(q)=\pi(q').$ Since $q$ and $q'$ can be joint by a geodesic,
there is a maximal torus of $X_1\times X_2$ containing both points
$q$ and $q'.$ Since the restriction of $\pi$ to this maximal torus is injective, we get
$q=q'.$ Thus $\pi$ is globally one-to-one and
$X$ splits as a product $X_1\times X_2.$
\end{proof}


\section{The Roots of a Symmetric Space with Rectangular Unit Lattice}
\label{SEC: Roots of symmetric spaces with rectangular unit lattice}
In this section we analyse the interplay between the unit lattice and the set of roots of symmetric spaces with rectangular unit lattice.\par

A symmetric space $X$ has \emph{rectangular} unit lattice $\Gamma_X\subset \lieA,$ if  there exists
an orthogonal basis $e_1,\dots, e_r$  of $\lieA$ whose $\Z$-span is $\Gamma_X.$
If $e_1,\dots, e_r$ can be chosen to have the same length, then $X$ is said to have
\emph{cubic} unit lattice. A symmetric space with rectangular unit lattice is necessarily compact as its maximal flats are compact.\par
Symmetric spaces with rectangular unit lattice include extrinsically symmetric space (see  \cite{EHQ} and the references therein), in particular hermitian symmetric spaces of compact type. We will see in Section \ref{SECT: rectangular unit lattice embedding} that every symmetric space with a rectangular unit lattice can be realized as an extrinsically symmetric space.\par

In the remainder of this section $X$ will always denote an indecomposable symmetric space of rank $r\geq 1$ with rectangular unit lattice $\Gamma_X,$
$e_1,\dots,e_r$ an orthogonal  basis of $\Gamma_X,$ and $\varepsilon_1,\dots, \varepsilon_r$ the corresponding dual basis of
$\lieA^*.$

\begin{proposition} \label{PROP: Weyl orbit}
The Weyl group $W_X$ leaves $\{\pm e_1,\dots, \pm e_r\}$ invariant and there are two possibilities:
\begin{enumerate}[(I)]
 \item $X$ is of compact type and $W_X$ acts transitively on $\{\pm e_1,\dots, \pm e_r\}.$
 \item The Euclidean factor in the universal cover of $X$ is 1-dimensional and
  the action of $W_X$ on $\{\pm e_1,\dots, \pm e_r\}$ has two different orbits which differ by a sign. Each orbit is a basis of $\Gamma_X$ and has
  the form $\{\lambda_1e_1,\dots, \lambda_re_r\}$ for suitable $\lambda_j\in\{-1,\; 1\}.$
\end{enumerate}
In particular, $e_1,\dots, e_r$ is always a cubic basis.
\end{proposition}

\begin{proof}
By orthogonality the set of shortest vectors in $\Gamma_X\setminus\{0\}$
is contained in the set $\{\pm e_1,\dots, \pm e_r\}.$ Thus $\{e_1,\dots, e_r\}$ contains
an element of shortest length in $\Gamma_X\setminus\{0\},$ say $e_1.$
Since $W_X$ acts isometrically on $\lieA$ leaving $\Gamma_X$ invariant,
$W_X(e_1)\cup W_X(-e_1)$ consists of shortest elements of $\Gamma_X\setminus\{0\}$
and is therefore contained in $\{\pm e_1,\dots, \pm e_r\}.$ Let
$J=\big\{j\in\{1,\dots, r\}:\; e_j\in W_X(e_1)\cup W_X(-e_1)\big\}.$
Then $\lieA=\lieA_1\oplus \lieA_2$ with $\lieA_1:=\mathrm{span}_{\R}\{e_j:\; j\in J\}$ and
$\lieA_2=\mathrm{span}_{\R}\big\{e_k:\; k\in\{1,\dots, r\}\setminus J\big\}.$ This
is a $W_X$-invariant orthogonal decomposition of $\lieA$ such that
$\Gamma_X=(\Gamma_X\cap\lieA_1)+
(\Gamma_X\cap\lieA_2).$ Since $X$ is indecomposable, Lemma \ref{LEM: splitting} yields $J=\{1,\dots, r\}.$
This shows
$W_X(e_j)\cup W_X(-e_j)=\{\pm e_1,\dots, \pm e_r\}$ for every $j\in\{1,\dots, r\}.$ Thus the action
 of $W_X$ on $\{\pm e_1,\dots, \pm e_r\}$ has at most two orbits.
 If there are two different orbits, then each orbit must be a basis of $\Gamma_X$ and the two orbits differ by a sign.\par
A positive dimensional Euclidean factor in the de Rham decomposition for the universal cover of $X$ corresponds to a nonzero linear
subspace $\lieA_0$ in $\lieA$ on which $W_X$ acts trivially.
If $W_X$ acts transitively on $\{\pm e_1,\dots, \pm e_r\},$ then the only fixed vector of $W_X$ in $\lieA$ is $0,$ that is
 $\lieA_0=\{0\}.$ In the other case we may assume (after suitably changing the signs) that $\{e_1,\dots, e_r\}$ is a $W_X$-orbit.
Then $\lieA_0=\R\cdot \sum_{j=1}^r e_j$ and the Euclidean factor in the de Rham decomposition of $\tilde{X}$ has dimension one.
\end{proof}

\begin{proposition}[c.\ f.\ {\cite[Lemma 7]{LoosCubic}}]
\label{PROP: possible roots}
 $2\Root_X\subset\{\pm \varepsilon_j,\; \pm 2\varepsilon_j,\; \pm (\varepsilon_j\pm \varepsilon_k):\; 1\leq j, k\leq r\}.$
\end{proposition}
\begin{proof}
 Let $\alpha\in\Root_X.$ Then
 $2\alpha=\sum\limits_{j=1}^r m_j\varepsilon_j$ for some $m_j\in\Z$ by (i) in Theorem \ref{THM: unit lattice sym space}.
Thus $\frac{1}{2}\alpha^\vee=\frac{2}{\sum\limits_{j=1}^r m_j^2}\sum\limits_{k=1}^r m_ke_k.$ By
(ii) in Theorem \ref{THM: unit lattice sym space} we get $\frac{2m_k}{\sum_{j=1}^r m_j^2}\in\Z$ for all $k=1,\dots, r$ and the
Lemma follows.
\end{proof}

\begin{proposition}\label{PROP: better basis for roots}
After suitable changes of signs for some of the $\varepsilon_i$ we have
$$\{\varepsilon_i -\varepsilon_j :\;  1\leq i < j\leq r\} \subset 2\Root_X.$$
\end{proposition}

\begin{proof}
Let $I=\{1,\dots ,r\}$. We show by induction on $k$ that for all $k\in I$ we have:
\begin{itemize}
\item[$(*)_k$] There exists $I'\subset I$ with $|I'|=k$ and $\varepsilon_i-\varepsilon_j\in 2\Root_X$ for all $i,j \in I'$, $i \neq j$, after a possible change of signs of some of the $\varepsilon_i$, $i \in I'.$
\end{itemize}
While $(*)_r$ proves the Lemma, $(*)_1$  is obvious. \\
Assume that $I'$ satisfies $(*)_k$ for some $k<r$. Then $\{\pm\varepsilon_i: i\in I'\}$ is not $W_X$-invariant according to Proposition \ref{PROP: Weyl orbit}. More precisely, there exists $\ell\in I'$ and $m\in I\setminus I'$ with $\varepsilon_\ell + \varepsilon_m\in 2\Root_X$ or  $\varepsilon_\ell - \varepsilon_m\in 2 \Root_X$ as $W_X$ is generated according to Lemma 3.2 by reflections of the form $s_{\varepsilon_i}$, $s_{\varepsilon_i +\varepsilon_j}$ or $s_{\varepsilon_i-\varepsilon_j}$, $i,j\in I$, $i \neq j$. Note that these generators only occur if the corresponding roots occur and that they leave $\{\pm\varepsilon_n\}$ invariant for all $n\in I$ except in case of $s_{\varepsilon_i \pm\varepsilon_j}$ which both interchange $\{\pm\varepsilon_i\}$ with $\{\pm\varepsilon_j\}$. After a possible change of sign of $\varepsilon_m$ we may assume $\varepsilon_\ell -\varepsilon_m\in 2 \Root_X$. But then also
$\varepsilon_i  -\varepsilon_m = s_{\varepsilon_i-\varepsilon_\ell}(\varepsilon_\ell-\varepsilon_m) \in 2\Root_X$ for all $i\in I'$. Thus $I'\cup \{m\}$ satisfies $(*)_{k+1}$  completing  the induction step.
\end{proof}

\begin{proposition}\label{PROP: root systems of rect unit lattice}
$\Root_X$ is one of the following classical root systems:
$$\begin{array}{lllr}
\Root_1&=&\frac{1}{2}\big\{\varepsilon_j-\varepsilon_k\, (1\leq j\neq k\leq r)\big\}& (A_{r-1}),\\
\Root_2&=&\frac{1}{2}\big\{\pm \varepsilon_j\, (1\leq j\leq r), \; \pm \varepsilon_j\pm\varepsilon_k\, (1\leq j<k\leq r)\big\}&
 (B_{r})
,\\
\Root_3&=&\frac{1}{2}\big\{\pm 2\varepsilon_j\, (1\leq j\leq r), \;  \pm \varepsilon_j\pm\varepsilon_k\, (1\leq j< k\leq r)\big\}
&(C_r),\\
\Root_4&=&\frac{1}{2}\big\{\pm \varepsilon_j\pm\varepsilon_k\, (1\leq j< k\leq r)\big\}&(D_r),\\
\Root_5&=&\frac{1}{2}\big\{\pm\varepsilon_j\, (1\leq j\leq r), \;\pm 2\varepsilon_j\, (1\leq j\leq r), \;  \pm \varepsilon_j\pm\varepsilon_k\, (1\leq j< k\leq r)\big\}& (BC_r).
 \end{array}$$
\end{proposition}

Except for the first case
$\Root_X$ is a root systems in $\lieA^*$ and $X$ is of compact type while
in the first case $\Root_X$ is a root system in the codimension-one subspace
$\big\{\alpha\in\lieA^*:\; \alpha(\sum_{j=1}^n e_j)=0\big\}$ of $\lieA^*$ and $X$ splits off locally an $S^1$-factor.
Except for $D_2$, $\Root_X$ is an irreducible root system in its span.

\begin{proof}
By Proposition \ref{PROP: better basis for roots} we may assume that $\varepsilon_j-\varepsilon_k\in 2\Root_X$ for all
$j,k\in I = \{1,\dots,r\}$ with $j\neq k$.
Now we distinguish two cases:
\begin{enumerate}[(i)]
\item There exists $j\in I$ with $\{\varepsilon_j, 2\varepsilon_j\}\cap 2\Root_X\neq\emptyset:$\\
If $\varepsilon_j\in2\Root_X$ or $2\varepsilon_j\in2\Root_X$ for one (and hence all)
$j\in I$ then  the reflections $s_{e_j}\in W_X$ allow us  to change the signs
of the $\varepsilon_j$ individually. Thus all $\varepsilon_j\pm\varepsilon_k\in2\Root_X$.
This yields the root systems of type $B_r$, $C_r$ and $BC_r$ for $r\geq 1$ depending on whether $\varepsilon_j$ or $2\varepsilon_j$ or both are contained in $2\Root_X$.
\item For all $j\in I$ we have $\{\varepsilon_j, 2\varepsilon_j\}\cap 2\Root_X=\emptyset:$\\
Then either $2\Root_X = A_{r-1} = \{\pm(\varepsilon_j-\varepsilon_k): j,k\in I, \,j\neq k\}$
or some (hence all) $\varepsilon_j+\varepsilon_k\in2\Root_X$ and $2\Root_X = D_r =
\{\pm(\varepsilon_j\pm\varepsilon_k): j,k\in I, \,j\neq k\}$.
\end{enumerate}
\end{proof}

\begin{note}The root systems of irreducible hermitian symmetric spaces of compact type ($C_r$ and $BC_r$) are determined in
\cite[Thm.\ 2, p.\ 362]{Moore}. A list of  root systems (and fundamental groups) of irreducible symmetric $R$-spaces
can be found in \cite[p.\ 305]{Take-84} which is based on an extensive study of symmetric $R$-spaces in \cite{Take-65} .
\end{note}

A \emph{Euclidean root datum} is a triple $(V,\Gamma,\Root)$ consisting of a Euclidean vector space $V,$ a full lattice $\Gamma$ in $V$ and a root system $\Root$ in a subspace of $V^*$ such that $\Gamma_0(\Root)\subset\Gamma\subset\Gamma_1(\Root)$ where $2\Gamma_0(\Root)$ is the $\Z$-span of $\Root^\vee$ and $2\Gamma_1(\Root)=\{v\in V: \alpha(v)\in\Z\; \text{for all}\; \alpha\in\Root\}.$ To each Riemannian symmetric space $Y$ one can associate a Euclidean root datum (up to isomorphism) by taking $V=\lieA$, $\Gamma=\Gamma_Y$ and $\Root=\Root_Y$
(see Appendix \ref{APPENDIX: unit lattice of compact sym space}).\par
Let $(V_r,\Gamma_L,\Root_j),\; j=1,\dots, 5,$ be the Euclidean root datum with $V_r$ an $r$-dimensional Euclidean vector space, $\Gamma_L$ a lattice
in $V_r$ spanned by an orthogonal basis $e_1,\dots, e_r$ of common length $L>0$ and $\Root_i$ one of the five root systems explicitly described in Proposition \ref{PROP: root systems of rect unit lattice} in terms of the dual basis $\varepsilon_1,\dots,\varepsilon_r$ of $e_1,\dots, e_r.$ We say that $X$ is of type $\widehat{A}_{r-1}$, $\widehat{B}_{r}$,
 $\widehat{C}_{r}$,  $\widehat{D}_{r}$ or $\widehat{BC}_{r}$ if its Euclidean root datum is up to scaling (that is up to a choice of $L$) isomorphic
 to $(V_r,\Gamma_L,\Root_j)$ with $\Root_j$ as in Proposition \ref{PROP: root systems of rect unit lattice} and of the corresponding type.

 \begin{theorem}\label{THM: possible root datum types}
  $X$ is of type $\widehat{A}_{r-1}$, $\widehat{B}_{r}$,
 $\widehat{C}_{r}$,  $\widehat{D}_{r}$ or $\widehat{BC}_{r}$ with $r\geq 2$ in case $\widehat{D}_{r}$ and $r\geq 1$ otherwise.
 The type of $X$ is unique and each type occurs. Moreover,
 $$\pi_1(X)\cong\begin{cases}
 \{1\}& \text{if $X$ is of type $\widehat{C}_r$ or $\widehat{BC}_r$}\\
 \Z_2  &  \text{if $X$ is of type $\widehat{B}_r$ or $\widehat{D}_r$}\\
 \Z  & \text{if $X$ is of type $\widehat{A}_{r-1}$}
 \end{cases}.$$
 \end{theorem}

\begin{proof} It follows from Propositions \ref{PROP: Weyl orbit} and \ref{PROP: root systems of rect unit lattice} that the Euclidean root datum of $X$ is isomorphic to $(V_r,\Gamma_L,\Root_j)$ for some $L>0$ and some $j=1,\dots, 5.$ Thus the type of $X$ is one of the above. Note that the types $\widehat{A}_{0}$ and $\widehat{D}_1$ coincide as $\Root_j=\emptyset$ in both cases. The examples
$\LieU_r,$ $\LieSO_{2r+1}$, $\LieSp_r$, $\LieSO_{2r}$ and $G_r(\C^{2r+1})$ show that all types occur. Since
$\pi_1(X)\cong\Gamma_X/\Gamma_0$ by Theorem \ref{THM: unit lattice sym space}, the statement on the fundamental groups follows by a simple calculation. Therefore the five type can be distinguished by their root systems and fundamental groups implying uniqueness.
\end{proof}

\begin{remark}
 Let $X$ be an indecomposable symmetric space with rectangular unit lattice. Then Case (II) occurs in
 Proposition \ref{PROP: Weyl orbit} if and only if $X$ is of type $\widehat{A}_{r-1}.$
\end{remark}

Let $Y$ be a symmetric space  and let $r$ be a positive integer. A submanifold $S$ of $Y$ is called \emph{polysphere}
of rank $r,$
 if $S$ is totally geodesic and isometric to a Riemannian product of $r$ round  spheres of dimension at least $2.$

\begin{theorem}[Polysphere theorem]
\label{THM: Polysphere}
Let $Y$ be a  compact symmetric space of rank $r\geq 1.$  Then the following statements are equivalent:
\begin{enumerate}[(i)]
 \item $Y$ contains a polysphere of rank $r.$
 \item $Y$ is simply connected and has rectangular unit lattice.
\end{enumerate}
\end{theorem}

\begin{proof}Assuming (i)
let $S$ be a polysphere of rank $r$ in $Y.$ Then the maximal torus $T$ of $S,$ which is a Riemannian product of circles, is also a maximal torus in $Y.$ Thus $Y$ has rectangular unit lattice. Now every closed curve in $Y$ is homotopic to a closed geodesic $\gamma$ in $T.$ Since $S$ is simply connected (being a product of spheres of dimension at least $2$), $\gamma$  is contractible in $S$ and hence in $Y.$
Thus $Y$ is simply connected.\par
If (ii) holds we may further assume that $Y$ is indecomposable. Then
$Y$ is of type $\widehat{C}_r$ or $\widehat{BC}_r$ by Theorem
\ref{THM: possible root datum types}.
In both cases the roots $\varepsilon_1,\dots, \varepsilon_r$ are strongly orthogonal, that
is  $\varepsilon_j\pm\varepsilon_k\notin R_Y\cup\{0\}$ for $j\neq k.$ Let $j\in\{1,\dots, r\}.$
Since in both cases
$2\varepsilon_j\notin\Root_Y$ we get  a
Lie subtriple $\lieS_j=\R\cdot H_{\varepsilon_j} + (\lieG_{\varepsilon_j}+\lieG_{-\varepsilon_j})\cap \lieP$ of $\lieP.$
Indeed, $[\lieS_j,\lieS_j]\subset \lieK_0\oplus (\lieG_{\varepsilon_j}+\lieG_{-\varepsilon_j})\cap \lieK$, $\lieK_0$ the centralizer of $\lieA$ in $\lieK,$ and $[\lieG_{\pm \varepsilon_j},\lieG_{\pm\varepsilon_j}]$ is orthogonal to all $H\in\lieA$ with $H\perp H_{\varepsilon_j}$.
The corresponding totally geodesic submanifold $S_j$ of $Y$ has non-vanishing constant sectional curvature and is therefore covered by a simply connected sphere.
Strong orthogonality implies that $\lieS_j$ and $\lieS_k$
are perpendicular  for $j\neq k$ and that $\lieS=\bigoplus\limits_{j=1}^r\lieS_j$ is again a Lie subtriple of $\lieP.$
The totally geodesic submanifold $S$
of $Y$ corresponding to $\lieS$
is covered by a product of $r$ simply connected spheres.
The root system $\Root_S=\{\pm\varepsilon_1,\dots, \pm\varepsilon_r\}$ of $S$ is a root subsystem of
$\Root_Y$ and the unit lattice of $S$ is $\Gamma_S=
\mathrm{span}_{\Z}\{e_1,\dots, e_r\}=\mathrm{span}_{\Z}\left\{\frac{1}{2}\varepsilon_1^\vee,\dots, \frac{1}{2}\varepsilon_r^\vee\right\}.$
By Theorem \ref{THM: unit lattice sym space} $S$ is simply connected and hence a polysphere of rank $r.$
\end{proof}

\begin{remark}
We see slightly more: The polyspheres in  Theorem \ref{THM: Polysphere} are actually `poly-Helgason-spheres',
that is they are products of maximal dimensional totally geodesic spheres  whose sectional curvature
realizes
the maximum of the sectional curvatures of $Y$
(see \cite{Helg-66}).\par
Theorem \ref{THM: Polysphere}  generalizes the `polysphere theorem' for hermitian symmetric spaces of compact type
(see e.g.\ \cite[p.\ 280]{Wolf-Hermitian}).
Note that not all simply connected Riemannian symmetric spaces with rectangular unit lattice are hermitian, e.\ g.\ the
compact symplectic groups $\LieSp_n,\;n\geq 2,$ the quaternionic Grassmannians or the Cayley plane.
\end{remark}

\section{Extrinsically Symmetric Embeddings}
\label{SECT: rectangular unit lattice embedding}

In this section we show Theorem \ref{THM: A}. As explained in the introduction, it is closely related to a result of \textsc{Loos} in \cite{LoosCubic}, but our arguments and our methods are entirely different. Let $X$ be a Riemannian symmetric space with rectangular unit lattice.

\begin{proof}[\indent\textsc{Proof of the existence statement in Theorem \ref{THM: A}}]
We may assume that $X$ is indecomposable and has positive dimension.
By Proposition \ref{PROP: Weyl orbit} we can choose an orthogonal basis $B=\{e_1,\dots, e_r\}$ of $\Gamma_X$
such that
\begin{enumerate}[(I)]
 \item $W_X(e_1)=B\cup(-B)$ if $X$ is of compact type, or
 \item $W_X(e_1)=B$ if $X$ splits off a local $\Sphere^1$-factor.
\end{enumerate}
We choose  a Weyl chamber $\lieA_+\subset\lieA$ of $X$ containing
$e_1$ in its closure. As before we extend $\lieA$ to a maximal abelian subalgebra $\lieT$ of $\lieG$ and
choose a Weyl chamber in $\lieT$ containing $\lieA_+$ in its closure. This induces a partial ordering on
$\lieT^*.$
Let $\{\varepsilon_1,\dots,\varepsilon_r\}\subset\lieA^*\subset\lieT^*$ be the basis dual to $B,$ that is
$\varepsilon_j(e_k)=0$ if $j\neq k$ and $\varepsilon_j(e_j)=1.$\par

By the Theorem of Cartan--Helgason--Takeuchi, Theorem \ref{THM: Cartan Helgason}, there exists a $K$-spherical $G$-module
$V$ with highest weight $\varepsilon_1,$ which is unique up to equivalence.
We now choose an element $v^o\in V^K\setminus\{0\}.$ This choice is unique up to a complex factor, since
$V^K$ has dimension one. We write
$v^o$  according to the weight space decomposition of $V$ (see Equation \ref{EQ: Weight space decomposition}) as
\begin{equation}
 \label{EQ: v expressed in weight vectors}
 v^o=\sum\limits_{\omega\in \Omega_V}v^o_{\omega}=\sum\limits_{\mu\in M}v^o_{\mu}
\end{equation}
with $v^o_{\mu}\in V_{\mu}$ and $M=\{\mu\in\Omega_V: v^o_{\mu}\neq 0\},$  where $\Omega_V$ is the set of weights for $V$ (see Appendix
\ref{SEC: Cartan-Helgason-Thm}).
From $K.v^o=\{v^o\}$ we see that $M$ is $W_X$-invariant and that $(K\cap\exp(\lieT)).v^o_{\mu}=v^o_{\mu}$ and
therefore $\mu(\lieT\cap\lieK)=0$ and $\mu(\Gamma_X)\subset \Z$
for all $\mu\in M.$ This shows that $M\subset \Gamma_X^*=\mathrm{span}_{\Z}(\varepsilon_1,\dots,\varepsilon_r).$
Since $\varepsilon_1$ is the highest weight of $V,$ we get
$M\setminus\{0\}\subset\{\pm \varepsilon_1,\dots, \pm\varepsilon_r\}.$
Now $\varepsilon_1\in M,$ since the projection of the highest weight space $V_{\varepsilon_1}$ onto $V^K$ is non-zero (see \cite[Lem.\ 3, p.\ 92]{TakeBook}).
As $M\setminus\{0\}$ is $W_X$-invariant we have
\begin{equation}\label{EQ: M minus 0}
\begin{array}{lllll}
 M\setminus\{0\}&=&\{\pm \varepsilon_1,\dots, \pm\varepsilon_r\} & \text{in Case (I)} & \text{or}\\
 M\setminus\{0\}&=&\{\varepsilon_1,\dots, \varepsilon_r\}& \text{in Case (II).} &
\end{array}
\end{equation}\par

We now consider the $G$-equivariant (and therefore smooth) map
\begin{equation}\label{EQ: embedding phi}
\Phi: X\to E:=\mathrm{span}_{\R}(G.v^o),\quad g.p\mapsto g.v^o\; .
\end{equation}
Let $(\cdot ,\cdot)_{\C}$ be a $G$-invariant hermitian inner product on $V,$ then the $G$-invariant real
inner product $\mathrm{Re}\big((\cdot ,\cdot)_{\C}\big)$ makes
$E$ a real Euclidean $G$-module.\par
For any $H\in\lieA$ we have
\begin{equation}\label{EQ: geodesics under Phi}
 \Phi(\exp(H).p)=\exp(H).v^o=
v^o_0+\sum\limits_{\mu\in M\setminus\{0\}}e^{2\pi i\mu(H)}v^o_{\mu}
\end{equation}
and by
Equation \eqref{EQ: M minus 0}
$$
\begin{array}{lllll}\Phi(\exp(H).p)-v^o_0&=&
\sum\limits_{j=1}^r (e^{2\pi i\varepsilon_j(H)}v^o_{\varepsilon_j}
+e^{-2\pi i\varepsilon_j(H)}v^o_{-\varepsilon_j})& \text{in Case (I)} & \text{or}\\
\Phi(\exp(H).p)-v^o_0&=&\sum\limits_{j=1}^r e^{2\pi i\varepsilon_j(H)}v^o_{\varepsilon_j}& \text{in Case (II).} &
\end{array}
$$
By Example \ref{EXAMPLE: CT} $\Phi(T_X)$ is a
Clifford torus.
The vectors $v^o_{\varepsilon_1}-v^o_{-\varepsilon_1},\dots, v^o_{\varepsilon_r}-v^o_{-\varepsilon_r}$ (resp.\ $v^o_{\varepsilon_1},\dots, v^o_{\varepsilon_r}$)
have the same length as $\varepsilon_1,\dots,\varepsilon_r$ lie in the same $W_X$-orbit. Thus
the unit lattice of $\Phi(T_X)$ is cubic. \par
Moreover, $\Phi|_{T_X}:T_X\to \Phi(T_X)$ is an isometry after possibly replacing the inner product on $E$ by a multiple:
In fact, the differential of $\Phi|_{T_X}$ at $p$ maps the cubic basis $\{e_1,\dots, e_r\}$ of $\Gamma_X$ onto a cubic basis
of the unit lattice of $\Phi(T_X).$ To show that $\Phi|_{T_X}$ is injective we assume that $\Phi(\exp(H).p)=\Phi(p)$ for some $H\in\lieA.$ Then $\varepsilon_j(H)\in\Z$ for all $j\in\{1,\dots, r\}.$ Thus $H\in\Gamma_X$ and $\exp(H).p=p.$\par
Since $\Phi$ is equivariant, it maps any maximal torus of $X$ isometrically onto a Clifford torus of $E.$ Therefore $\Phi$ is an isometric embedding of $X$ with the property that every geodesic of $\Phi(X)$ is contained in a totally geodesic Clifford torus. By  Theorem
\ref{THM: Clifford type Extr Sym} $\Phi(X)$ is an extrinsically symmetric space in $E.$\par
Finally,   the affine hull $\aff(\Phi(X))$ of $\Phi(X)$ coincides with $E.$ Otherwise $0\notin \aff(\Phi(X))$ and there would be a unique
nonzero element $w$ of minimal length in $\aff(\Phi(X)).$  Since $\aff(\Phi(X))$ is $G$-invariant, $G.w=w,$ a contradiction to
irreducibility.
\end{proof}

\begin{remark}\label{REM: after phi Clifford type}\leavevmode
Following \textsc{Cartan} \cite{Cartan} we can embed $V$ $G$-equivariantly into the space
 $C^{\infty}(X)$ of complex valued smooth  functions
 on $X$ by $h:V\to C^{\infty}(X),\; v\mapsto h_v,$
 where $h_v:X\to \C,\; x\mapsto (v,\Phi(x))_{\C}.$
 The submodule $h(V)$ of $C^{\infty}(X)$ contains the $K$-invariant function $h_{v^o}.$ After multiplying $v^o$ with a suitable constant,
 $h_{v^o}$  is given by
 $$
\begin{array}{lllll}
h_{v^o}(\exp(H).p)&=&a+\sum\limits_{j=1}^r \Big(e^{2\pi i \varepsilon_j(H)}
+e^{-2\pi i \varepsilon_j(H)}\Big)&&\\
&=&a+2\sum\limits_{j=1}^r \cos(2\pi  \varepsilon_j(H))& \text{in Case (I)} & \text{or}\\
h_{v^o}(\exp(H).p)&=&a+\sum\limits_{j=1}^r e^{-2\pi i \varepsilon_j(H)}& \text{in Case (II),} &
\end{array}$$ for all $H\in\lieA$
(see Equation \eqref{EQ: geodesics under Phi}). Here $a$ is a real positive constant such that
$h_{v^o}$ is perpendicular to the constant functions with respect to the $L^2$-inner product, that is
the integral of $h_{v^o}$ over $X$ vanishes. It can be shown that $a=0$ in Case (II) since $0$ is not a weight as it is not
contained in the convex hull of the Weyl orbit of $\varepsilon_1.$
\end{remark}

To finish the proof of Theorem \ref{THM: A} we next show:

\begin{theorem}[Uniqueness statement of  Theorem \ref{THM: A}]
\label{THM: uniqueness}
Let $X$ be a symmetric space with rectangular unit lattice and
 let $\Phi:X\to E$ and $\tilde{\Phi}:X\to\tilde{E}$ be two
 full, isometric, $G$-equivariant embeddings of $X$ into real $G$-modules whose images are extrinsically symmetric spaces.
 Then there exists a $G$-equivariant linear isometry $F:E\to \tilde{E}$ such that
 $F\circ\Phi=\tilde{\Phi}.$\par
 Moreover, if $X$ is indecomposable, then $E$ is irreducible and its complexification
 $E^{c}$  (as a $G$-module) is equivalent to $V$ with $V\cong\overline{V}$ if $X$ is of compact type and
 to $V\oplus \overline{V}$ with $V\not\cong\overline{V}$ if $X$ splits off a local $\Sphere^1$-factor, where
 $V$ is the irreducible complex $G$-module with highest weight $\varepsilon_1.$
\end{theorem}

\begin{proof}
By Proposition \ref{PROP: CT intrinsic implies extrinsic splitting} we may assume that $X$ is indecomposable. By Theorem \ref{THM: Clifford type Extr Sym} $\Phi(X)$ is a submanifold of Clifford type.
Given a full, isometric, G-equivariant embedding $\Phi:X\to E$  in a (real) Euclidean G-module
 $E,$ such that $\Phi(X)$ is of Clifford type, we recover the representation data of $E$ as follows.
 With the chosen base point $p\in X$ we set
 $v^o:=\Phi(p)\in E\setminus\{0\}.$ Note that $v^o$ is $K$-invariant and $\Phi(X)=G.v^o.$
 Now $E^c=E\otimes \C$ is a complex $G$-module.
As in Equation \eqref{EQ: v expressed in weight vectors} we decompose $v^o$ into weight vectors of $E^c$
as $v^o=\sum\limits_{\mu\in M}v^o_{\mu}$
where $v^o_{\mu}$ lies in the weight space in $E^c$ corresponding to the weight
$\mu\in\Omega_{E^c}$ of $E^c$ and $M=\{\mu\in\Omega_{E^c}: v^o_{\mu}\neq 0\}.$
Let $\mu\in M,$ then for all $H\in\lieT$ we have
$\exp(H).v^o_{\mu}=\sum\limits_{\mu\in M}e^{2\pi i\mu(H)}v^o_{\mu}$
for all $H\in\lieT.$ If $\exp(H)\in K$ then $\exp(H)v^o=v^o.$ Hence
$\exp(H)v^o_\mu=v^o_\mu$ and therefore $\mu(H)\in \Z$ for all $\mu\in M.$
This implies $\mu(\lieT\cap\lieK)=0$ and $\mu\in\Gamma_X^*.$\par

Since $\Phi$ is isometric and $\Phi(X)$ is of Clifford type,
$\Phi(T_X)$ is a Clifford torus in $E$
by Theorem \ref{THM: Clifford type Extr Sym}.
Let $\{e_1,\dots, e_r\}$ be an orthogonal basis of $\Gamma$ such that the curves
$$c_j:\R\to E,\quad t\mapsto \Phi(\exp(te_j).p)=\sum\limits_{\mu\in M}e^{2\pi it \mu(e_j)}v^o_\mu,$$
$j=1,\dots, r,$ have smallest positive period $1$ and parameterize the generating circles of $\Phi(T_X)$ through $v^o.$
In particular $c'_j(0)$ is perpendicular to $c'_k(0)$ for $j\neq k.$
By Proposition \ref{PROP: Weyl orbit} we may assume that
$W_X.e_1=\{\pm e_1,\dots, \pm e_r\}$ if $X$ is of compact type, or
$W_X.e_1=\{e_1,\dots, e_r\}$ if $X$ splits off a local $\Sphere^1$-factor.
We fix again a Weyl chamber $\lieA_+$ in $\lieA$ containing $e_1$ in its closure,
a Weyl chamber $\lieT_+$ in $\lieT$ containing $\lieA_+$ in its closure and
the partial ordering on $\lieT^*$ induced by $\lieT_+.$
Let $j\in\{1,\dots, r\}.$
Since $c_j$ has smallest positive period $1$ and since $c_j(\R)$ is a planar circle,
there are vectors $a_j,b_j, d_j\in E$ with
$c_j(t)=d_j+\cos(2\pi t)a_j+\sin(2\pi t)b_j$ for all $t\in\R.$
It follows that $-4\pi^2 c'_j(0)= c'''_j(0)$ and hence $\mu(e_j)=\mu(e_j)^3$ for all $\mu\in M.$ Thus for any $\mu\in M$ and for any
$j\in\{1,\dots, r\}$ we have $\mu(e_j)\in\{-1,0,1\}.$
We extend the inner product on $E$ to a $G$-invariant Hermitian inner product on $E^c.$ Then
$v^o_\mu$ is perpendicular to $v^o_{\mu'}$ for $\mu,\mu'\in M$ with $\mu\neq \mu'.$
Since different generating circles of $\Phi(T_X)$ through $v^o$ lie in perpendicular
planes, we have
$$0=\langle c_j''(0),c''_k(0)\rangle=16\pi^4\cdot \sum\limits_{\mu\in M}
(\mu(e_j))^2(\mu(e_k))^2\|v^o_\mu\|^2.$$
If $k\in\{1,\dots,r\}$ and $\mu\in M$ with $\mu(e_k)\neq 0$ (i.\ e.\ $\mu(e_k)=\pm 1$) then $\mu(e_j)=0$ for all $j\neq k.$
Summing up we have
$$M\setminus\{0\}\subset\{\pm \varepsilon_1,\dots, \pm \varepsilon_r\},$$
where $\varepsilon_1,\dots,\varepsilon_r\in\lieA^*\subset\lieT^*$ is the dual basis of $e_1,\dots, e_r.$
Note that the only possible elements in the intersection of $M$ with the closure of $\lieA_+^*$ are
$\varepsilon_1$ and $\overline{\varepsilon_1}:=w_o(-\varepsilon_1),$ where
$w_o$ is the unique element in $W_X$ with $w_o(-\lieA_+)=\lieA_+.$\par

Since $\Phi(X)=G.v^o$ is full in $E$, $\Phi(X)$ is not contained in any proper complex affine subspace of
$E\otimes \C.$ Let $V'$ be an irreducible  complex $G$-submodule of $E\otimes \C$. Then
the orthogonal projection of $v^o$ onto $V'$ is nonzero.
Otherwise $G.v^o$ would be contained in the orthogonal complement of $V'.$ Thus $V'$ is $K$-spherical.
Let $\lambda'\in (\Gamma_X^*)_+$ be the highest weight of
$V'$ (see Theorem \ref{THM: Cartan Helgason}), then $\lambda'\in (M\setminus\{0\})\cap (\Gamma_X^*)_+
\subset\{\varepsilon_1,\overline{\varepsilon}_1\}.$
This shows that $E\otimes\C$ is a direct sum of $K$-spherical $G$-modules equivalent to
$V(\varepsilon_1)$ or $V(\overline{\varepsilon_1})$ in the notation of Theorem \ref{THM: Cartan Helgason}.
Note that $V(\overline{\varepsilon_1})=\overline{V(\varepsilon_1)},$ the conjugate representation of $V(\varepsilon_1).$
Since $\overline{E\otimes\C}$ is equivalent to $E\otimes \C,$ both $V(\varepsilon_1)$ and
$\overline{V(\varepsilon_1)}$ occur in $E\otimes \C.$ But, by fullness of $\Phi(X),$ none of the equivalence classes of these
modules appears  twice  in $E\otimes \C.$
Indeed, assume for simplicity that $E\otimes \C=U\oplus U',$
where $U$ and $U'$ are isomorphic irreducible complex $G$-modules. Let $f:U\to U'$ be a $G$-equivariant isomorphism and
let $v^o=u+u'$ with $u\in U\setminus\{0\}$ and $u'\in U'\setminus\{0\}.$ Then $u'=zf(u)$ for some $z\in\C\setminus\{0\},$
since ${U'}^K$ has dimension one. But then
$\{u+zf(u): u\in U\}$ is a proper $G$-submodule of $E\otimes \C$ containing $\Phi(X),$
contradicting the fullness of $\Phi(X).$\par

Summing up we have seen that $E\otimes \C\cong V(\varepsilon_1)$ if $\varepsilon_1=\overline{\varepsilon}_1,$ that is if $X$ is of compact
type or $E\otimes \C\cong V(\varepsilon_1)\oplus \overline{V(\varepsilon_1)}$ if $\varepsilon_1\neq\overline{\varepsilon}_1,$ that is if $X$ splits off a local $\Sphere^1$-factor. Thus $E\otimes \C$ is (up to equivalence) uniquely determined by $X.$ Since
the realification of $E\otimes\C$ is equivalent to $E\oplus E,$ $X$ also determines $E$ uniquely (up to equivalence).
Moreover $E$ must be irreducible, since otherwise $E\otimes \C$ would split into two non-trivial self-conjugate submodules.
If $E\otimes \C$ is irreducible (that is if $X$ is of compact type) then $E$ is a real form of
$E\otimes \C$ and therefore $\mathrm{dim}_{\R}(E^K)=\mathrm{dim}_{\C}((E\otimes \C)^K)=1.$
If $X$ splits off a local $\Sphere^1$-factor, then $E$ is equivalent to the realification of $V(\varepsilon_1)$ (which is equal to the realification of $\overline{V(\varepsilon_1)}$) and $\mathrm{dim}_{\R}(E^K)=2\cdot \mathrm{dim}_{\C}(V(\varepsilon_1))=2.$
In this case $E$ carries a $G$-invariant complex structure $J$ and $E^K=\R v^o+\R Jv^o.$
In any case the $G$-module $E$ is uniquely determined by $X$.\par

Let $\tilde{\Phi}:X\to\tilde{E}$ be another full, isometric, G-equivariant embedding in a Euclidean $G$-module
$\tilde{E}.$ Then $\tilde{E}$ and $E$ must be equivalent. We may therefore assume that $E=\tilde{E}.$ Since
$\tilde{v}^o:=\tilde{\Phi}(o)\in E^K,$ there exists a $G$-equivariant isomorphism $f:E\to E$ of the form $f=a\cdot \mathrm{id}$
(if $X$ is of compact type) or $f=a\cdot \mathrm{id}+b\cdot J$ with $a,b\in\R,$ and $f(v^o)=\tilde{v}^o.$ Therefore
$f\circ\Phi=\tilde{\Phi}.$ Since $\Phi$ and $\tilde{\Phi}$ are isometries and $\Phi(X)$ and $\tilde{\Phi}(X)$ are both full, $f$ is a linear isometry.
\end{proof}

\begin{remark}
The proof of Theorem  \ref{THM: uniqueness} shows that in the context of this statement $E$ is irreducible and its complexification
 $E^{\C}$ is (as a $G$-module) equivalent to $V$ with $V\cong\overline{V}$ if $X$ is of compact type and
 to $V\oplus \overline{V}$ with $V\not\cong\overline{V}$ if $X$ splits off a local $\Sphere^1$-factor, where
 $V$ is the irreducible complex $G$-module with highest weight $\varepsilon_1.$
\end{remark}

\begin{remark}
 Let $X$ be a simply connected indecomposable symmetric space with rectangular unit lattice of rank $r\geq 1,$
 let $S$ be a totally geodesic polysphere (see Theorem \ref{THM: Polysphere}) that contains a maximal torus $T_X$ of $X$
 and let $\Phi$ the embedding of Theorem \ref{THM: A}. Then $\Phi(S)$ splits extrinsically as a product of round 2-spheres in $3$-dimensional affine Euclidean subspaces.\par
Indeed, since the maximal torus $\Phi(T_X)$ of $\Phi(S)$ is a Clifford torus,
 $\Phi(S)$ is of Clifford type. Thus the
intrinsic splitting of $S$ into 2-spheres is also extrinsic by Proposition \ref{PROP: CT intrinsic implies extrinsic splitting}. Now the uniqueness statement of Theorem \ref{THM: A} yields the claim.
\end{remark}

The uniqueness statement in Theorem \ref{THM: A} has the following immediate consequence which is a sharpening of \cite{EQT19} without using classification and case-by-case studies.

\begin{corollary}\label{COR: Extension of isometries}
 Let $X_1\subset E_1$ and $X_2\subset E_2$ be two compact extrinsically symmetric spaces (or, equivalently, two submanifolds of Clifford type).
 Then any isometry $f:X_1\to X_2$ extends to a Euclidean map between the affine hulls of $X_1$ and $X_2.$
 In particular every (intrinsic) isometry of a compact extrinsically symmetric space $X\subset E$ extends to a Euclidean motion of the Euclidean space $E.$
\end{corollary}

\section{Eigenvalues of the Laplacian}
\label{SEC: Eigenvalues of the Laplacian}

The knowledge of the Euclidean root datum of a symmetric space $X$ with rectangular unit lattice from Section \ref{SEC: Roots of symmetric spaces with rectangular unit lattice} allows us to compute the eigenvalues of the Laplacian of $X$ by means of \textsc{Freudenthal's} formula.
We conclude in particular that the embedding of $X$ constructed in Theorem \ref{THM: A} is congruent to  an
embedding in the first eigenspace in general up to some exceptional cases.
A proof of this last result has been outlined earlier by \textsc{Kobayashi}--\textsc{Takeuchi} \cite{KobTak} and \textsc{Ohnita} \cite{Ohnita} using the classification of standardly embedded symmetric $R$-spaces.\par

Let $X$ be an indecomposable symmetric space with rectangular unit lattice $\Gamma_X\subset\lieA.$
As in Section
\ref{SEC: Roots of symmetric spaces with rectangular unit lattice} let
$e_1,\dots, e_r$ be a cubic basis of $\Gamma_X$ contained in the $W_X$-orbit of
$e_1$ and let $\varepsilon_1,\dots, \varepsilon_r$ be the corresponding dual basis.
By Section
\ref{SEC: Roots of symmetric spaces with rectangular unit lattice}
we choose a set of positive roots $\Root_X^+$ in $\Root_X$ such that
$$2\Root_X^+ =2\Root_X\cap \{\varepsilon_j,\; 2\varepsilon_j,\; \varepsilon_j+\varepsilon_k,\; \varepsilon_j-\varepsilon_k:\; 1\leq j<k\leq r\}.$$
Let $2\rho_X$ be the sum of the positive roots of $X$ counted with their \emph{multiplicities}, that is the (complex) dimension of their root spaces. Let $m_1$, $m_2$, $m_+$ and $m_-$ denote the multiplicities of $\varepsilon_j$, $2\varepsilon_j$,
$\varepsilon_j+\varepsilon_k$ and $\varepsilon_j-\varepsilon_k,$ respectively  (which are independent of $j$ and $k$), $1\leq j<k\leq r.$
Any of them might be zero and $m_+=m_-:=m_{\pm},$ unless $X$ is of type $\widehat{A}_{r-1}$ or $\widehat{D}_2.$
With $\sum\limits_{1\leq j<k\leq r}(\varepsilon_j+\varepsilon_k)=(r-1)\sum\limits_{j=1}^r\varepsilon_j$ and
$\sum\limits_{1\leq j<k\leq r}(\varepsilon_j-\varepsilon_k)=\sum\limits_{j=1}^r(r-2j+1)\varepsilon_j$
Proposition \ref{PROP: root systems of rect unit lattice}
implies:

\begin{lemma}\label{LEM: 2rho}
 $$2\rho_X=\begin{cases}
           \frac{m_-}{2}\sum\limits_{j=1}^r(r-2j+1)\varepsilon_j & \text{in case}\; \widehat{A}_{r-1},\\
           \sum\limits_{j=1}^r \left(\frac{m_1}{2}+m_2+m_{\pm}(r-j)\right)\varepsilon_j & \text{in cases}\; \widehat{B}_{r},\; \widehat{C}_{r}, \widehat{D}_r\; (r\geq 3)\; \text{and}\; \widehat{BC}_r,\\
           \frac{m_+ +m_-}{2}\varepsilon_1+\frac{m_+ -m_-}{2}\varepsilon_2, &  \text{in  case}\; \widehat{D}_2.
          \end{cases}$$
\end{lemma}

Lemma \ref{LEM: 2rho} enables us to compute the eigenvalues of the Laplacian $\Delta=-\mathrm{div}\circ \mathrm{grad}$ of $X.$ In fact, the complex eigenspaces $E_{\lambda}(\C)$ of $\Delta$  decompose into irreducible
submodules which, according to Theorem \ref{THM: Cartan Helgason}, are (up to equivalence) the $V_{\omega}$ with $\omega\in (\Gamma^*_X)_+$ and
each $V_{\omega}$ occurs in precisely one complex eigenspace $E_{\lambda}(\C).$ This $\lambda$, which we denote by $\lambda_{\omega}$ is given by
\textsc{Freudenthal}'s formula
\begin{equation}
 \label{EQ: Freudenthal}
\lambda_{\omega}=4\pi^2 \langle \omega+2\rho_X,\; \omega\rangle
\end{equation}
(cf.\ \cite[Cor.\ 2, p.\ 191]{TakeBook}). Since $\Delta$ is a real operator, each real eigenspace $E_{\lambda}(\R)$ is a real form of $E_{\lambda}(\C)$
and $\{\lambda_{\omega}:\; \omega\in (\Gamma^*_X)_+\}$ is the set of eigenvalues of $\Delta.$
Combining Lemma \ref{LEM: 2rho} with \textsc{Freudenthal}'s formula \eqref{EQ: Freudenthal}
and observing that $[0,1]\to X,\; t\mapsto \exp(t e_j)p,$ are shortest closed geodesic arcs of $X$ of common length
$L=\|e_j\|=\frac{1}{\|\varepsilon_j\|}$ $(j=1,\dots, r)$ we get

\begin{proposition}\label{PROP: lambda-omega}
Let $\omega=\sum\limits_{j=1}^r k_j\varepsilon_j\in(\Gamma^*_X)_+.$ Then
  $$\lambda_{\omega}=\left\{\begin{array}{l}
           \frac{4\pi^2}{L^2}\sum\limits_{j=1}^rk_j\left(k_j+\frac{m_-}{2}(r-2j+1)\right)\; \text{in case}\; \widehat{A}_{r-1},\\
          \frac{4\pi^2}{L^2}\sum\limits_{j=1}^rk_j\left(k_j+\frac{m_1}{2}+m_2+m_{\pm}(r-j)\right)\; \text{in  cases}
          \; \widehat{B}_{r},\; \widehat{C}_{r}, \widehat{D}_r\; (r\geq 3)\; \text{and}\; \widehat{BC}_r,\\
           \frac{4\pi^2}{L^2}\left(k_1\left(k_1+\frac{m_++m_-}{2}\right)+k_2\left(k_2+\frac{m_+-m_-}{2}\right)\right)\; \text{in case}\; \widehat{D}_2.
          \end{array}\right.$$
\end{proposition}

If $\omega=\sum\limits_{j=1}^r k_j\varepsilon_j$ with $k_1,\dots, k_r\in\Z,$ then $\omega\in (\Gamma_X^*)_+ $ if and only if $k_1\geq \dots\geq k_r$
in
case $\widehat{A}_{r-1};$ $k_1\geq \dots \geq k_r\geq 0$ in cases $\widehat{B}_r$, $\widehat{C}_r$, $\widehat{BC}_r$ and
$k_1\geq \dots \geq k_{r-1}\geq |k_r|$ in case $\widehat{D}_r.$ From this we get the following corollary which for irreducible $X$
is equivalent to a result of \textsc{Ohnita} \cite[Thm.\ 7]{Ohnita}.

\begin{corollary}\label{COR: embedding into eigenspaces of Laplacian}
 The embedding of an indecomposable symmetric space $X$ with rectangular unit lattice given in Theorem \ref{THM: A} is congruent to an embedding in the first eigenspace of the Laplacian of $X,$
 unless $X$ is of type $\widehat{A}_{r-1},\; r\geq 2,$ with $m_->2$ or of type $\widehat{D}_2$ with $|m_+-m_-|> 2.$
\end{corollary}

\begin{remark}
 In the exceptional cases the embedding is not congruent to an embedding in the first eigenspace but rather to an embedding in the $k$-the eigenspace of the Laplacian, where $k$ might be arbitrary large if only $m_-$ or $|m_+-m_-|$ are sufficiently large.
\end{remark}

\begin{proof}
 Since the embedding of Theorem \ref{THM: A} maps $X$ into $V_{\varepsilon_1}$, the question remains
 whether $\lambda_{\varepsilon_1}$ is the first positive eigenvalue of $\Delta.$ If $X$ is not of type $\widehat{A}_{r-1}$ of $D_2$
 the answer is apparently `yes'. In the following we assume $L=2\pi$ for simplicity. If $X$ is of type $\widehat{A}_{r-1}$
 then
 $$\lambda_{\omega}=\sum\limits_{j=1}^rk_j^2+\frac{m_-}{2}\sum\limits_{j=1}^{\left[\frac{r}{2}\right]}(k_j-k_{r+1-j})(r+1-2j).$$
 Therefore $\lambda_{\omega}\geq \lambda_{\varepsilon_1}$ if $k_1-k_r\geq 1.$ Thus $\lambda_{\varepsilon_1}$ is minimal if and only if
 $1+\frac{m_-}{2}(r-1)\leq \lambda_{\sum_{j=1}^r\varepsilon_j}=r$ which is equivalent to $m_-\leq 2.$ \par
 If $X$ is of type $\widehat{D}_2$ we may assume $m_-\leq m_+$ since replacing $\varepsilon_2$ by $-\varepsilon_2$ replaces $m_+$ by $m_-.$
 Then Proposition \ref{PROP: lambda-omega} together with $k_1\geq |k_2|$ implies
 $\lambda_{\omega}\geq k_1^2+k_2^2+k_1m_-$ and thus
 $\lambda_{\omega}\geq 2+m_-=\lambda_{\varepsilon_1-\varepsilon_2}$ if $k_1\geq 2$ or $k_1=|k_2|=1.$ Therefore
 $\lambda_{\varepsilon_1}$ is minimal if and only if $\lambda_{\varepsilon_1}\leq \lambda_{\varepsilon_1-\varepsilon_2}$
 which is equivalent to $m_+-m_-\leq 2.$
\end{proof}

\appendix
\section{The unit lattice of a compact symmetric space}
\label{APPENDIX: unit lattice of compact sym space}

In this appendix we extend well known relations between unit lattices and roots systems for symmetric spaces of compact type (see e.\ g.\ \cite[Prop.\ 2.4, p.\ 68 f. and Thm.\ 3.6, p.\ 77]{Loos2}) to all compact symmetric spaces. \par
Let $X$ be a symmetric space that covers a compact symmetric space, $p\in X,$ and $F_X$ a maximal flat of $X$ containing $p.$
Let $G=G(X)$, $\pi:G/K\to X$, $\lieG=\lieK\oplus\lieP$, $\lieA\subset \lieP$, $\Gamma_X\subset\lieA$, $\Root_X\subset\lieA^*$ and
$\Root_X^\vee\subset\lieA$ as described in Section \ref{SEC: preliminaries}.
Let $Z$ be the center of $G.$ Then
\begin{equation}\label{EQ: exp in Z(G)}
\{H\in\lieA:\; \exp(H)\in Z\}=\{H\in\lieA: \alpha(H)\in\Z \; \text{for all}\; \alpha\in \Root_X\},
\end{equation}
since for $H\in\lieA$, $\exp(H)\in Z$ is equivalent to $\Ad(\exp(H))=e^{\ad(H)}=\mathrm{id}$ and thus to $e^{\ad(H)}|_{\lieG_\alpha}=\mathrm{id}$ for all $\alpha\in\Root_X,$ that is $e^{2\pi i\alpha(H)}=1.$
Let
$$\Gamma_0:=\Gamma_0(\Root_X):=\mathrm{span}_{\Z}\frac{1}{2}\Root_X^\vee$$ and
$$\Gamma_1:=\Gamma_1(\Root_X):=\frac{1}{2}\{H\in\lieA:\; \alpha(H)\in\Z\; \text{for all}\; \alpha\in\Root_X\}.$$
If $X$ is of compact type $\Gamma_0$ and $\Gamma_1$ are full lattices in $\lieA.$ However, if $X$ is compact, but not of compact type, then
$\Gamma_0$ only spans $\lieA_0^\perp\subset\lieA$, $\lieA_0=\bigcap_{\alpha\in\Root_X}\mathrm{ker}(\alpha),$ and
$\Gamma_1$ is not discrete as $\lieA_0\subset\Gamma_1.$ Note that in contrast to $\Gamma_X$, $\Gamma_0$ and $\Gamma_1$ do not change if $X$ is replaced by a symmetric space that covers $X$ or that is covered by $X$ (as $\Root_X$ does not change).

\begin{theorem}
\label{THM: unit lattice sym space}\leavevmode
Let $X$ be a symmetric space that covers a compact symmetric space. Then:
\begin{enumerate}[(i)]
  \item $2R_X\subset\Gamma^*_X=\{\alpha\in\lieA^*:\; \alpha(H)\in\Z\; \text{for all}\; H\in\Gamma_X\}.$
  \item $R^\vee_X\subset 2\Gamma_X.$
  \item $\Gamma_0\subset\Gamma_X\subset\Gamma_1.$
\item The fundamental group $\pi_1(X)$ of $X$ is isomorphic to $\Gamma_X/\Gamma_0.$
 \end{enumerate}
\end{theorem}

\begin{proof}
 \begin{enumerate}[(i)]
  \item Let $H\in\Gamma_X.$ From $\exp(H)\in K\subset G^{\sigma}$ we get $\exp(-H)=\sigma(\exp(H))=\exp(H)$ and thus $\exp(2H)=e.$ The statement now follows from Equation \eqref{EQ: exp in Z(G)}.
  \item The unit lattice $\Gamma_{\widetilde{X}}$ of the universal cover $\widetilde{X}$ of $X$ is contained in $\Gamma_X.$ Now $\widetilde{X}$ splits as $\R^n\times X',$ where $X'$ is a simply connected symmetric space of compact type. Thus $\Gamma_{X'}=\Gamma_{\tilde{X}}\subset
  \Gamma_X.$ Since $\Root_{X'}=\Root_{\widetilde{X}}=\Root_X$ and $\Gamma_0(\Root_X')=\Gamma_{X'}$ (see \cite[Corollary, p.\ 77]{Loos2}),
  the claim follows.
  \item This is just a reformulation of (i) and (ii).
\item The fundamental group $\pi_1(X)$ is isomorphic to the group $\Delta$ of deck transformations of the universal covering
$\widetilde{X}\cong \R^n\times X'\to X,$  where $X'$ is a compact simply connected symmetric space. Let $\tilde{p}\in\widetilde{X}$
be an element in the fibre over $p.$
As
$\Gamma_{\tilde{X}}=\Gamma_{X'}$ and $\Gamma_{X'}=\Gamma_0$ by \cite[Corollary, p.\ 77]{Loos2}. Thus the maximal flat $F_{\widetilde{X}}$ of $\widetilde{X}$ corresponding to $\lieA$ with
$\tilde{p}\in F_{\widetilde{X}}$
is isomorphic to $\lieA/\Gamma_0.$ The universal covering restricts to a  covering of the maximal flats
$\lieA/\Gamma_0\cong F_{\widetilde{X}}\to T_X\cong\lieA/\Gamma_X$
which has deck transformation group $\Gamma_X/\Gamma_0.$ Thus it is sufficient to show that
$\Delta$ leaves $F_{\widetilde{X}}$ invariant. As the elements of $G(\widetilde{X})$ push down to isometries of $X$, $G(\widetilde{X})$
normalizes $\Delta$ and thus actually centralizes $\Delta,$ as $\Delta$ is discrete and $G(\widetilde{X})$ is connected.
Now let $\delta\in \Delta$ and $q\in F_{\widetilde{X}}.$
Then there exists a maximal flat $F'_{\widetilde{X}}$ of $\widetilde{X}$ containing  $q$ and $\delta(q)$
and $\tau\in G(\widetilde{X})$ with $\tau(q)=q$ and $\tau(F'_{\widetilde{X}})=F_{\widetilde{X}}.$
Thus $\delta(q)=\delta(\tau(q))=\tau(\delta(q))\in\tau(F'_{\widetilde{X}})=F_{\widetilde{X}}.$
 \end{enumerate}
\end{proof}

A triple $(V,\Gamma,\Root)$ is called a \emph{Euclidean root datum} if $V$ is a Euclidean vector space, $\Gamma$ a full lattice in $V,$
and $\Root$
a root system in a subspace of $V^*,$ such that
$\Gamma_0(\Root)\subset\Gamma\subset\Gamma_1(\Root),$ where
$\Gamma_0(\Root)$ and
$\Gamma_1(\Root)$ are defined above.
Two Euclidean root data $(V,\Gamma,\Root)$ and $(V',\Gamma',\Root')$ are called \emph{isomorphic} if there exists a linear isometry $\varphi:V\to V'$
with $\varphi(\Gamma)=\Gamma'$ and $\varphi(\Root)=\Root',$ where $\varphi(\alpha)=\alpha\circ \varphi^{-1}$
for $\alpha\in V^*.$ For more details see \cite{HQ}. With these definitions Theorem \ref{THM: unit lattice sym space}
implies

\begin{corollary}\label{COR: root datum of X}
 Let $X=G/K$ be a compact symmetric space, $\lieG=\lieK\oplus\lieP$, $\lieA\subset\lieP$, $\Gamma_X\subset\lieA$ and $\Root_X\subset\lieA^*$
 as above. Then $(\lieA,\Gamma_X,\Root_X)$ is a Euclidean root datum. Its isomorphism class only depends on $X$ and is called the \emph{Euclidean root datum of $X$}. The the  fundamental group of $X$ is isomorphic to $\Gamma_X/\Gamma_0$ and encoded in the Euclidean root datum.
\end{corollary}

\section{The Theorem of Cartan--Helgason--Takeuchi}
\label{SEC: Cartan-Helgason-Thm}

Let $X$ be a compact symmetric space and $p$ a base point in $X.$ In this section we denote by $G$ the transvection group of $X,$
that is $G=G(X).$ Let $K$ be the stabilizer of $p$ in $G$ and let $\lieG=\lieK\oplus\lieP$ be the corresponding decomposition
of the Lie algebra $\lieG$ of $G.$
Retaining  the notation of Section \ref{SEC: preliminaries}
we choose a maximal abelian subspace $\lieA$ in $\lieP$ and denote by $T_X=\exp(\lieA).p$ the corresponding maximal torus of
$X$ and by $\Root_X\subset\lieA^*$ the corresponding set of roots of $X.$
We choose a \emph{Weyl chamber} $\lieA_+$ of $X$ in $\lieA,$
that is a connected component of $\lieA\setminus\big(\bigcup_{\alpha\in \Root}\mathrm{ker}(\alpha)\big).$
Note that $W_X$ acts transitively on the set of all Weyl chambers in $\lieA.$
Such a choice of $\lieA_+$ gives rise to a partial ordering on $\lieA^*=\mathrm{Hom}(\lieA,\R)$ by setting
\begin{equation}\label{EQ: partial ordering}
\omega_1<\omega_2:\iff (\omega_1(H)< \omega_2(H)\;
\text{for all}\: H\in\lieA_+),
 \end{equation}
for $\omega_1,\omega_2\in\lieA^*.$ .\par
We extend $\lieA$ to a maximal abelian subalgebra $\lieT$
of $\lieG.$ Then $\lieT=\lieA\oplus (\lieT\cap \lieK).$
One defines the set of roots of $\lieG$ with respect to $\lieT,$
the Weyl chambers of $\lieG$ in $\lieT$ and the Weyl group of $\lieG$ with respect to $\lieT$ similar to the
corresponding notions for $X$ (see e.g.\ \cite{Helg-78, Loos2}).
We now choose a  Weyl chamber $\lieT_+$ of $\lieG$ in $\lieT$
with the property that our previously fixed
$\lieA_+$ lies in the closure of $\lieT_+.$
As in Equation \eqref{EQ: partial ordering}, $\lieT^+$ induces a partial ordering on $\lieT^*=\Hom(\lieT,\R).$
Using the bi-invariant metric on $\lieG$ we identify
$\lieA$ with $\lieA^*$ and $\lieT$  with $\lieT^*.$
Note that this identification of $\lieA^*$ with $\lieA$ is $W_X$-equivariant and that $\lieA^*$ becomes  the subset
$\{\mu\in\lieT^*: \mu(\lieT\cap\lieK)=\{0\}\}$ of $\lieT^*,$
and the partial orderings on $\lieA^*$ induced by
$\lieA_+$ and by $\lieT_+$  coincide. We denote by $\lieA^*_+$ the image of $\lieA_+$ in $\lieT^*$  under this identification
and by $\mathrm{cl}(\lieA^*_+)$ its closure. We set
$\Gamma_X^*=\{\omega\in \lieA^*: \omega(\Gamma_X)\subset\Z\}$ and
$(\Gamma_X^*)_+=\Gamma_X^*\cap \mathrm{cl}(\lieA^*_+).$
If $\Gamma_G=\{H\in\lieT: \exp(H)=e\}$ and $\Gamma_G^*=\{\omega\in \lieT^*: \omega(\Gamma_G)\subset\Z\}$
and $(\Gamma_G^*)_+=\Gamma_G^*\cap \mathrm{cl}(\lieT^*_+)$ then (cf.\ \cite[Lem.\ 1, p.\ 143]{TakeBook})
$$(\Gamma_X^*)_+\subset (\Gamma_G^*)_+.$$
\par
Let $V$ be a complex $G$-module.
We endow $V$ with a $G$-invariant hermitian inner product $(\cdot ,\cdot)_{\C}.$ For an element $\omega\in \lieT^*$
we set $$V_{\omega}:=\left\{v\in V: \exp(H).v=e^{2\pi i\omega(H)}v\; \text{for all}\; H\in\lieT\right\}.$$
The elements of $\Omega_V:=\{\omega\in\lieT^*: V_{\omega}\neq \{0\}\}$
are called \emph{weights} of $V$ and the decomposition
\begin{equation}
\label{EQ: Weight space decomposition}
 V=\bigoplus\limits_{\omega\in\Omega_V}V_{\omega}
\end{equation}
is called the \emph{weight space decomposition} of $V.$
It is orthogonal with respect to $(\cdot ,\cdot)_{\C}.$
We have $\Omega_V\subset \Gamma_G^*.$ Note that the roots of a compact Lie algebra $\lieG$ are the weights of the adjoint representation of the group of inner automorphisms of $\lieG.$
\par
A \emph{$K$-spherical $G$-module} is an irreducible complex $G$-module $V$ such that $V^K:=\{v\in V: k.v=v\}$ has positive dimension.
By a result of \textsc{Cartan} $V^K$ has dimension one \cite[I.5 and V.17]{Cartan} (see also \cite[Thm.\ 5.5, p.\ 92]{TakeBook}).
We denote by  $\widehat{G}_K$ set of all equivalence classes of $K$-spherical $G$-modules.

\begin{theorem}[Cartan--Helgason--Takeuchi, {\cite[Thm.\ 8.2]{TakeBook}}]
\label{THM: Cartan Helgason}
 The map
 $$(\Gamma_X^*)_+\to \widehat{G}_K,\quad \omega\mapsto [V(\omega)]$$
 is a bijection. Here $V(\omega)$ denotes an irreducible complex $G$-module with highest weight $\omega$
 and $[V(\omega)]$ its equivalence class.
\end{theorem}

\begin{note}
Theorem \ref{THM: Cartan Helgason} has a long history (c.\ f.\ \textsc{Borel} \cite[Chap.\ IV, B, \S 4]{Borel-Hist}).
After previous work by \textsc{Cartan} \cite{Cartan}, \textsc{Helgason} \cite{Helg-70} (see also \cite[Ch.\ V, \S 4, 1.]{Helg-84}) gave a proof for semi-simple symmetric spaces.
The extension to all compact symmetric spaces is due to  \textsc{Takeuchi} \cite{TakeBook}.
\end{note}


\section{Submanifolds of Clifford Type and their splittings}

Let $r>0$ be an integer and $\rho_1,\dots,\rho_r$ be positive real numbers. A space
$$C':=\{z=(z_1,\dots, z_r)\in\C^r:\; |z_j|=\rho_j\; \text{for all}\; j=1,\dots, r\}$$
is called a \emph{standard Clifford torus}.
Endowed with the inner product
$\langle z,w\rangle=\mathrm{Re}\Big(\sum\limits_{j=1}^r z_j\overline{w_j}\Big)$ $\C^r$ is a Euclidean space.\par
A submanifold $C$ of a Euclidean space $E$ is called a \emph{Clifford torus} if it is the image of a standard
Clifford torus $C'\subset\C^r$ under an affine isometric map $\varphi:\C^r\to E.$
Let $p=\varphi(w)\in C.$ Then the unit lattice
$\Gamma=\{v\in T_pC:\; \Exp_p(v)=p\},$ where $\Exp_p$ denotes the Riemannian exponential map of $C$ at $p,$
 has an orthogonal basis $b_1,\dots, b_r,$ which is unique up to sign and order (see \cite[Lemma 1.1]{EHQ}). Indeed, $\Gamma$ is the direct sum of the unit lattices of the
circles $C_k(p)=\varphi(\{z\in\C^r:\; |z_k|=\rho_k\; \text{and}\; z_j=w_j\; \text{for all}\; k\neq j\}),\; k=1,\dots, r,$
which lie in the pairwise orthogonal affine planes $E_k(p)=\varphi\{z\in\C^r:\; z_j=w_j\; \text{for all}\; k\neq j\}.$

\begin{example}
\label{EXAMPLE: CT}
Let $V$ be a Euclidean space. We consider
$V^c = V\otimes\C$ as a \emph{real} Euclidean space by taking the real part of the hermitian extension of the inner product on $V.$
Let $v,w\in V^c$ be orthogonal with respect to the hermitian inner product and $v\neq 0.$
Then $C = \{zv+\bar zw: z\in\C\; \text{with}\; |z|=1\}$ is a planar circle in $V^c.$
In fact, for $z = c+is$ with $c = \cos t$, $s =\sin t$ we have
$$zv+\bar zw = (c+is)v+(c-is)w = c(v+w) + s\;i(v-w)$$
which is a circle since $v+w$ and $i(v-w)$ have the same length and
are perpendicular with respect to the (real) Euclidean inner product on $V^c$.\par
More generally, if $v_1,\dots,v_r,w_1,\dots,w_r\in V^c$
are pairwise orthogonal with respect to the hermitian inner product and $v_1,\dots v_r$ are all nonzero,
then $$C:=\left\{\sum\limits_{j=1}^r( z_jv_j+\bar z_jw_j):\;  z_j\in\C\; \text{with}\; |z_j|=1\right\}$$
is a Clifford torus in $V^c.$
\end{example}

As in \cite{EHQ} we call an embedded submanifold $X\subset E$ of a Euclidean space $E$ \emph{of Clifford
type}, if every geodesic of $X$ lies in a totally geodesic submanifold $C$ of $X$ which is
a Clifford torus in $E.$

\begin{theorem}[{\cite[Corollary 2.8]{EHQ}}]
\label{THM: Clifford type Extr Sym}
 Let $X\subset E$ be a connected submanifold of a Euclidean space $E.$ Then $X$ is a compact extrinsically symmetric space
 if and only if $X$ is of Clifford type. In this case, $X$ is intrinsically a compact symmetric space with rectangular unit lattice and all maximal tori of $X$ are  Clifford tori in $E.$
\end{theorem}

\subsection*{Intrinsic and extrinsic splittings}

Let $X$ be a connected Riemannian manifold. Recall that $X$ \emph{splits intrinsically} as $X_1\times \dots\times X_k$
if $X_1,\dots, X_k$ are Riemannian manifolds and there exists an isometry $f:X_1\times \dots \times X_k\to X.$ Such an isometry $f$
is called an (explicit) splitting and gives rise to $k$ pairwise orthogonal foliations whose leaves through
$p=f(x_1^0,\dots, x_k^0)$ are $X_j(p):=f(x_1^0,\dots, x_{j-1}^0, X_j,x_{j+1}^0, \dots, x_k^0)$ for $1\leq j\leq k.$\par
Let $Y$ be another connected Riemannian manifold and $\iota:X\to Y$ an isometric embedding. We say that the submanifold $\iota(X)\subset Y$  \emph{splits extrinsically} as
$X_1\times \dots\times X_k$ if there exists a commutative diagram
$$\begin{CD}
Y_1\times\dots\times Y_k @>F>> Y\\
@AA\iota_1\times \dots\times \iota_kA @AA\iota A\\
X_1\times \dots \times X_k @>f>> X
\end{CD}$$
where the first and the last rows are (explicit) intrinsic splittings and the maps $\iota_j:X_j\to Y_j$, $j=1,\dots, k$, are isometric embeddings. In this case the foliations corresponding to intrinsic splittings of $X$ and $Y$ satisfy
$\iota(X_j(p))\subset Y_j(\iota(p))$ for all $p\in X$ and for all $1\leq j\leq k.$

\begin{lemma}
 \label{LEM: intrinsic splitting extrinsic}
 Let $\iota:X\to Y$ be an isometric embedding and let $f:X_1\times\dots\times X_k\to X$ and $F:Y_1\times \dots\times Y_k\to Y$ be explicit intrinsic splitting that satisfy $\iota(X_j(p))\subset Y_j(\iota(p))$ for all $p\in X$ and for all $1\leq j\leq k.$
 Then $\iota(X)$ splits extrinsically as $X_1\times \dots \times X_k.$ More precisely, there exist isometric embeddings $\iota_j:X_j\to Y_j$, $j=1,\dots, k,$ such that $F\circ (\iota_1\times \dots\times \iota_k)=\iota\circ f.$
\end{lemma}

\begin{proof}
 For simplicity we may assume $X=X_1\times\dots\times X_k$, $Y_1=Y_1\times \dots\times Y_k$, $f=\mathrm{id}_X$ and $F=\mathrm{id}_Y.$
 Let $\iota_j'=\mathrm{pr}_j\circ \iota,$ where $\mathrm{pr}_j:Y\to Y_j$ is the $j$-th projection.  Since for instance
 $X_1(p)\subset Y_1(\iota(p)$ for all $p=(x_1,\dots, x_k)\in X$, we get
 $$\iota(X_1,x_2,\dots,x_k)\subset (Y_1,\iota'_2(x_1,\dots, x_k),\dots, \iota'_k(x_1,\dots, x_k).$$
 Since the left hand side is independent of $x_1$ and not contained in any other leaf $Y_1(q),$ the maps
 $\iota_2',\dots ,\iota_k'$ are independent of $x_1.$ Similarly, $\iota'_j$ is independent of $x_{\ell}$ for all
 $\ell\neq j.$ Thus every $\iota_j'$ induces a mapping $\iota_j:X_j\to Y_j$ and $\iota=\iota_1\times \dots\times \iota_k.$
 \end{proof}

 We immediately get

 \begin{corollary}\label{COR: extr splitting}
  Let $X\subset E$ be an embedded submanifold of a Euclidean space $E$ and $f:X_1\times X_k\to X$ an intrinsic splitting of $X.$
  If
  \begin{enumerate}[(i)]
   \item $\aff(X_j(p))$ is perpendicular to $\aff(X_\ell(p))$ and
   \item $\aff(X_j(p))$ is parallel to $\aff(X_j(q)$
  \end{enumerate}
  for all $p,q\in X$ and $1\leq j<\ell<k$ then $X$ splits extrinsically as $X_1\times \dots \times X_k$ and vice-versa. Here $\aff(M)$ denotes the affine hull of a subset $M\subset E.$
 \end{corollary}

 For Clifford tori we get:

\begin{lemma}\label{LEM: Clifford tori splitting}
If a Clifford torus $C\subset E$ splits intrinsically as $X_1\times \cdots \times X_k$, then this splitting is also extrinsic
and $X_j(p)$ are Clifford tori for all $j=1,\dots k$ and all $p\in C.$
\end{lemma}

\begin{proof} It suffices to assume $k=2$ and $C=X_1\times X_2.$ Let $p\in C.$
 Then the unit lattice $\Gamma\subset T_pC$ of $C$ at $p$ has an orthogonal
 basis $b_1,\dots, b_r$ with $|b_1|\leq |b_2|\leq \dots\leq |b_r|.$
 Since  $X_1(p))$ and $X_2(p)$ are compact totally geodesic submanifolds of
 $C,$ they are both tori and we get an orthogonal decomposition of the unit lattice
 $\Gamma$ of $C$ at $p$ into $\Gamma^1\oplus \Gamma^2,$ where $\Gamma^1$ is the unit lattice of $X_1(p))$
 and $\Gamma^2$ the unit lattice of $X_2(p).$ Let $j\in\{1,\dots, r\}.$ From
 $b_j=b_j^1+b_j^2$ with $b_j^1\in\Gamma^1$ and $b_j^2\in\Gamma^2$ we get
 $|b_j|^2=|b_j^1|^2+|b_j^2|^2.$ Thus $b_1\in\Gamma^1$ or $b_1\in\Gamma^2,$ since $b_1$ is a nonzero element
 of shortest length in $\Gamma$ (see also part (ii) in the proof of Proposition \ref{PROP: Weyl orbit}). By induction we get $b_j\in\Gamma^1\cup\Gamma^2$ for all $j\in\{1,\dots, r\}.$
 Hence $\Gamma^1$ and $\Gamma^2$ are rectangular lattices and the $X_1(p))$ resp.\ $X_2(p)$ are Clifford tori
 in orthogonal subspaces $E_1(p)$ resp.\  $E_2(p).$ Since $E_j(p)$ is parallel to $E_j(q)$ for all $j=1,2$ and $p,q\in C,$ $C$ splits extrinsically by Corollary \ref{COR: extr splitting}.
\end{proof}

Lemma \ref{LEM: Clifford tori splitting} generalizes to submanifolds of Clifford type:

\begin{proposition}\label{PROP: CT intrinsic implies extrinsic splitting}
 Let $X\subset E$ be a submanifold of Clifford type. If $X$ splits intrinsically as a Riemannian product
 $X_1\times \dots\times X_k,$ then this splitting is extrinsic.
\end{proposition}
\begin{proof}
Since $X$ is a compact symmetric space, the same holds for all factors $X_j.$
 It suffices to assume $k=2$ and $X=X_1\times X_2.$ Then for a point $p=(p_1,p_2)\in X$ we have
 $X_1(p)=X_1\times \{p_2\}$ and $X_2(p)=\{p_1\}\times X_2.$
We set
 $E_1(p)=\aff(\{p_1\}X_1(p))$ and
  $E_2(p)=\aff(X_2(p)).$
  Now for any two points in an $X_j$ there exists a maximal torus of $X_j$ containing these two points. Thus we have
  $$E_1(p)=\aff\bigcup_{T_1\ni p_1} (T_1\times \{p_2\}),$$ where the union is taken over all maximal tori of $X_1$ that contain $p_1$ and
  $$E_2(p)=\aff\bigcup_{T_2\ni p_2} (\{p_1\}\times T_2),$$
  where the union is taken over all maximal tori of $X_2$ that contain $p_2.$\par
  If $T_j,\; j=1,2,$ are maximal tori of $X_j$ with $p_j,p'_j\in T_j$ then
  $T=T_1\times T_2$ is a maximal torus of $X$ and hence a Clifford torus in $E.$
  By Lemma \ref{LEM: Clifford tori splitting} and Corollary \ref{COR: extr splitting} the affine hulls
  of $T_1\times \{p_2\}$ and $\{p_1\}\times T_2$ are perpendicular while those of
  $T_1\times \{p_2\}$ and $T_1\times \{q_2\}$  (resp.\ $\{p_1\}\times T_2$ and  $\{q_1\}\times T_2$)
  are parallel. Thus
  $E_1(p)$ is perpendicular to $E_2(p)$ for all $p$ and $E_1(p)$ is parallel to $E_1(q)$ (resp.\
  $E_2(p)$ is parallel to $E_2(q)$) for all $p,q\in X.$ The assertion now follows from Corollary \ref{COR: extr splitting}.
\end{proof}



\end{document}